\documentclass[11pt,reqno]{amsart}

\usepackage{xcolor}
\usepackage[T1]{fontenc}
\usepackage{amsmath}							
\usepackage{amssymb}
\usepackage{amsthm}
\usepackage{amscd}
\usepackage{amsfonts}
\usepackage{mathabx}
\usepackage{stmaryrd}
\usepackage[all]{xy}

\usepackage{caption}
\usepackage{subcaption}
\usepackage{euler}

\usepackage{extarrows}

\usepackage[colorlinks, linktocpage, citecolor = purple, linkcolor = blue]{hyperref}
\usepackage{color}

\usepackage{tikz}									
\usetikzlibrary{matrix}
\usetikzlibrary{patterns}
\usetikzlibrary{positioning}
\usetikzlibrary{decorations.pathmorphing}
\usetikzlibrary{cd}

\usepackage{ytableau}
\usepackage{fullpage}
\usepackage[shortlabels]{enumitem}

\newtheorem{theorem}{Theorem}[section]

\newtheorem{proposition}[theorem]{Proposition}

\theoremstyle{definition}
								
\newtheorem{definition}[theorem]{Definition}
\newtheorem{example}[theorem]{Example}

\newtheorem*{remark*}{Remark}

\newcommand{\ZZ}{\mathbb{Z}}
\newcommand{\CC}{\mathbb{C}}

\newcommand{\QQ}{\mathbb{Q}}
\newcommand{\RR}{\mathbb{R}}

\newcommand{\te}{\widetilde{e}}

\newcommand{\tv}{\widetilde{v}}

\newcommand{\tF}{\widetilde{F}}
\newcommand{\tG}{\widetilde{G}}

\newcommand{\tX}{\widetilde{X}}

\newcommand{\tGa}{\widetilde{\Gamma}}

\newcommand{\Ga}{\Gamma}

\newcommand{\De}{\Delta}

\newcommand{\ph}{\varphi}
\newcommand{\rt}{\mathrm{rt}}

\DeclareMathOperator{\Ker}{Ker}
\DeclareMathOperator{\Ram}{Ram}

\DeclareMathOperator{\Hom}{Hom}
\DeclareMathOperator{\Sym}{Sym}

\DeclareMathOperator{\Vor}{Vor}

\let\Im\relax
\DeclareMathOperator{\Im}{Im}

\newcommand{\FS}{\mathrm{FS}}

\newcommand{\st}{\mathrm{st}}

\DeclareMathOperator{\Pic}{Pic}

\DeclareMathOperator{\trop}{trop}
\DeclareMathOperator{\val}{val}

\DeclareMathOperator{\Id}{Id}
\DeclareMathOperator{\Div}{Div}

\DeclareMathOperator{\Prym}{Prym}

\DeclareMathOperator{\GL}{GL}
\DeclareMathOperator{\Supp}{Supp}

\DeclareMathOperator{\Jac}{Jac}

\title{Resolution of the Prym map in genus 4}
 
\author{Dmitry Zakharov}
\address{Department of Mathematics, Central Michigan University, Mount Pleasant, MI 48859, USA}
\email{\href{mailto:dvzakharov@gmail.com}{dvzakharov@gmail.com}}

\begin{document}

\begin{abstract}

    We use the tropical trigonal construction to describe the minimal toroidal resolution of the Prym period map $\overline{R}_4\dashrightarrow\overline{A}_3$.


\end{abstract}

\maketitle
\setcounter{tocdepth}{1}
\tableofcontents

\section{Introduction}

There are two standard ways of assigning a principally polarized abelian variety (ppav) to an algebraic curve $X$ of genus $g$. First, the Jacobian variety $\Jac(X)$ is a ppav of dimension $g$. Second, given a connected \'etale double cover $\tX\to X$, the Prym variety $\Prym(\tX/X)$ is a ppav of dimension $g-1$. Varying these constructions in moduli, we obtain the Torelli and Prym period maps
\[
t_g:M_g\to A_g,\quad p_g:R_g\to A_{g-1}
\]
between the relevant moduli spaces.

The moduli spaces $M_g$ and $R_g$ have standard compactifications $\overline{M}_g$ and $\overline{R}_g$, whose boundaries consist of strata
\[
\overline{M}_g=\coprod M(G),\quad \overline{R}_g=\coprod R(\tG/G),
\]
parametrizing respectively stable curves of genus $g$ with fixed dual graph $G$ and admissible double covers with fixed dual cover $\tG\to G$. The moduli space $A_g$ has a family of toroidal compactifications $\overline{A}_g^{\tau}$, determined by a choice of admissible decomposition $\tau$ of the rational closure $\Omega^{\mathrm{rt}}_g$ of the cone of positive definite quadratic forms $\Omega_g$. Three standard toroidal compactifications are known: the perfect cone compactification $\overline{A}_g^P$, Igusa's central cone compactification $\overline{A}_g^C$, and the second Voronoi compactification $\overline{A}_g^V$. These three compactifications coincide for $g=2$ and $g=3$, and we denote them by $\overline{A}_2$ and $\overline{A}_3$. The period maps extend to rational maps
\[
\overline{t}^{\tau}_g:\overline{M}_g\dashrightarrow \overline{A}^{\tau}_g,\quad \overline{p}^{\tau}_g:\overline{R}_g\dashrightarrow \overline{A}^{\tau}_{g-1},
\]
and for each $\tau$ we may pose three natural questions, in order of increasing difficulty.

The first question is to describe the indeterminancy loci of the extended period maps. The answer is fully known for the extended Torelli map. Mumford and Namikawa~\cite{1976Namikawa} showed that the Torelli map extends to the second Voronoi compactification, while Alexeev and Brunyate~\cite{2012AlexeevBrunyate} showed that it extends to the perfect cone compactification. On the other hand,~\cite{2012AlexeevBrunyate} and~\cite{2012Alexeevetal} showed that the Torelli map extends to the central cone compactification if and only if $g\leq 8$, disproving a conjecture of Namikawa.

The corresponding problem for the Prym map $\overline{p}^{\tau}_g:\overline{R}_g\dashrightarrow \overline{A}^{\tau}_{g-1}$ is significantly more difficult. Friedman and Smith~\cite{1986FriedmanSmith} constructed examples of boundary strata which do not admit extensions of the Prym period map to any toroidal compactification of $A_g$. Generalizing these examples,~\cite{2002AlexeevBirkenhakeHulek} and~\cite{2002Vologodsky} showed that the indeterminancy locus of $\overline{p}_g^V:\overline{R}_g \dashrightarrow \overline{A}_{g-1}^V$ is the union of the closures of certain \emph{Friedman--Smith} boundary strata $\FS_n$ for $n=2,\ldots,g-1$. The paper~\cite{2017Casalaina-MartinGrushevskyHulekLaza} showed that the indeterminancy locus of $\overline{p}_g^P$ contains $\overline{\mathrm{FS}_2}\cup \overline{\mathrm{FS}_3}$ and does not intersect the strata $\mathrm{FS}_n$ for $n\geq 4$ (but may intersect their closures). The PhD thesis~\cite{2018Frinak} showed that, up to a locus of codimension at least 10, the indeterminancy locus of $\overline{p}_5^P$ is in fact $\overline{\mathrm{FS}_2}\cup \overline{\mathrm{FS}_3}$.

The second natural question is to describe the blowups necessary to resolve the indeterminancy loci. The moduli spaces $\overline{R}_g$ and $\overline{A}_{g-1}^{\tau}$ are locally quotients of toroidal schemes by finite group actions. In the neighborhood of a boundary stratum $R(\tG/G)$ indexed by a double cover $\tG\to G$, the rational map $\overline{p}^{\tau}_g$ is determined by a map of cones $\phi:\RR_{\geq 0}^{E(G)}\to \Omega^{\mathrm{rt}}_{g-1}$ sending the generators of $\RR_{\geq 0}^{E(G)}$ to certain rank one quadratic forms given by the Picard--Lefschetz formula. The period map $\overline{p}^{\tau}_g$ extends to $R(\tG/G)$ if and only if $\phi(\RR_{\geq 0}^{E(G)})$ lies in a cone of the admissible decomposition $\tau$. More generally, any rational decomposition of $\RR_{\geq 0}^{E(G)}$ into subcones having this property determines a toric blowup that resolves the Prym period map $\overline{p}^{\tau}_g$ along $R(\tG/G)$. There is a minimal such decomposition, obtained by pulling back the cones of $\tau$ along $\phi$.

An unpublished result of Alexeev states that a simple blowup is sufficient to resolve the Prym map $\overline{p}^V_g:\overline{R}_g\dashrightarrow \overline{A}_{g-1}^V$ along each Friedman--Smith stratum $\mathrm{FS}_n$ (but not necessarily their closures). Vologodsky~\cite{2004Vologodsky} found the necessary resolutions for a family of boundary strata lying in the closure of $\mathrm{FS}_2$. The paper~\cite{2017Casalaina-MartinGrushevskyHulekLaza} computed the cone decompositions for a number of examples (including the perfect and central cone compactifications), and showed that blowing up $\overline{\mathrm{FS}_2}$ is in fact sufficient to fully resolve the Prym map $\overline{p}_3:\overline{R}_3\dashrightarrow\overline{A}_2$. 

In this paper, we find the minimal toroidal resolution of the extended Prym map $\overline{p}_4:\overline{R}_4\dashrightarrow\overline{A}_3$ (whose indeterminancy locus by~\cite{2002AlexeevBirkenhakeHulek} and~\cite{2002Vologodsky} is $\overline{\mathrm{FS}}_2\cup\overline{\mathrm{FS}}_3$) by describing, for each double cover of weighted graphs $\tG\to G$ with $g(G)=4$, the minimal decomposition of the cone $\RR_{\geq 0}^{E(G)}$ that resolves $\overline{p}_4$ along the stratum $R(\tG/G)$. Our description is based on several ideas from tropical geometry. First, we interpret $\phi:\RR_{\geq 0}^{E(G)}\to \Omega^{\mathrm{rt}}_{g-1}$ as a \emph{tropical Prym period map}: we view a point of $\RR_{\geq 0}^{E(G)}$ as a double cover of metric graphs $\tGa\to \Ga$ with model $\tG\to G$, and its image under $\phi$ as its \emph{tropical Prym variety} $\Prym_c(\tGa/\Ga)$. By a theorem of Cools and Draisma~\cite{2018CoolsDraisma}, every metric graph of genus $g(\Ga)=4$ is \emph{trigonal}, in other words admits (after a tropical modification) a degree 3 harmonic morphism $\Ga'\to \De$ to a matric tree. This determines the \emph{trigonal decomposition} of $\RR_{\geq 0}^{E(G)}$: on each subcone, all metric graphs have a trigonal structure of the same combinatorial type.

We then use the \emph{tropical trigonal construction} of R\"ohrle and the author~\cite{2025RoehrleZakharovA} to express the tropical Prym variety $\Prym_c(\tGa/\Ga)$ as the tropical Jacobian $\Jac(\Pi)$ of a tetragonal graph $\Pi$, whose combinatorial type is constant on the trigonal cones.  Mumford and Namikawa's extension result for $\overline{t}_g^V$ then directly implies that the trigonal cone decomposition resolves the Prym--Torelli map. The minimal decomposition resolving the Prym map can then be obtained by comparing the Voronoi polytopes in each trigonal cone, and turns out to agree with a decomposition conjectured by the author in~\cite{2025Zakharov}.

We conclude by briefly discussing the third natural question, which is to find \emph{modular} interpretations of the extended period maps and their resolutions. While the compactifications $\overline{M}_g$ and $\overline{R}_g$ are modular by definition, only one (to the best of the author's knowledge) of the toroidal compactifications of $\overline{A}_g$ is modular: Alexeev showed in~\cite{2002Alexeev} that the second Voronoi compactification $\overline{A}_g^V$ is the normalization of the main irreducible component of the moduli space $\overline{\mathrm{AP}}_g$ of stable semiabelic pairs. Furthermore, in~\cite{2004Alexeev} he gave a modular interpretation of the corresponding extended Torelli map $\overline{M}_g\to \overline{\mathrm{AP}}_g$ in terms of compactified Jacobians. It would be interesting to find a similar interpretation of the resolved Prym period map, but this question is beyond the scope of this paper. 

Most of the calculations in this paper rely on Sage code that I wrote for the earlier paper~\cite{2025Zakharov}. AI was used for a small part of the code but not for any proofs. I would like to sincerely thank thank Melody Chan, Samuel Grushevsky, Yoav Len, Yelena Mandelshtam, Felix R\"ohrle, Raluca Vlad, and Martin Ulirsch for insightful discussions.

\section{Graphs, double covers, and gonality}

In this section we recall the necessary definitions concerning weighted metric graphs, harmonic morphisms, tropical gonality, and tropical ppavs.

\subsection{Graphs and weighted graphs}

We consider finite graphs, possibly with loops and multiedges, and all graphs are connected unless stated otherwise. We denote the vertex and edge sets of a graph $G$ by $V(G)$ and $E(G)$, respectively. The \emph{valence} $\val(v)$ of a vertex is the number of incident edges, with loops counted twice.

A \emph{weighted graph} $(G,g)$ is a graph $G$ together with a function $g:V(G)\to \ZZ_{\geq 0}$ called the \emph{vertex genus}. Vertex genera play a primarily bookkeeping role and do not affect monodromy calculations. The \emph{genus} of a weighted graph is
\[
g(G)=b_1(G)+\sum_{v\in V(G)}g(v)=|E(G)|-|V(G)|+1+\sum_{v\in V(G)}g(v).
\]
The \emph{Euler characteristic} of a vertex $v\in V(G)$ of a weighted graph $G$ is the quantity
\[
\chi(v)=2-2g(v)-\val(v).
\]
A vertex $v\in V(G)$ is called \emph{stable} if $\chi(v)<0$, \emph{semistable} if $\chi(v)=0$, and \emph{unstable} if $\chi(v)>0$. A weighted graph $(G,g)$ is \emph{stable} if all of its vertices are stable, and \emph{semistable} if all of its vertices are stable or semistable. The \emph{stabilization} $G^{\st}$ of a weighted graph $(G,g)$ of genus $g(G)\geq 2$ is obtained by iteratively removing unstable vertices $v\in V(G)$ (deleting all extremal trees with no vertices of positive genus), and then removing all semistable vertices by joining the two incident edges into a single edge. 

A \emph{metric graph} $\Ga=(G,\ell)$ is a graph $G$ together with an assignment of positive real \emph{lengths} $\ell:E(G)\to \RR_{>0}$ to the edges of $G$, and a \emph{weighted metric graph} is a weighted graph with edge lengths. We do not directly use chip-firing on metric graphs and do not review it here.

\subsection{Harmonic morphisms and ramification}

A \emph{morphism} $f:\tG\to G$ of graphs sends vertices to vertices and edges to edges in a manner that preserves adjacency. A \emph{harmonic morphism} of graphs is a morphism $f:\tG\to G$ together with a \emph{local degree} function $d_f:V(\tG)\cup E(\tG)\to \ZZ_{>0}$, such that for any vertex $\tv\in V(\tG)$ and any edge $e\in E(G)$ rooted at $f(\tv)$ we have
\[
d_f(\tv)=\sum d_f(\te),
\]
where the sum is taken over all edges of $\tG$ rooted at $\tv$ and mapping to $e$. If $G$ is connected, then a harmonic morphism $f:\tG\to G$ has a well-defined \emph{global degree} equal to
\begin{equation*}
\label{eq:harmonicglobaldegree}    
\deg(f)=\sum_{\tv\in f^{-1}(v)}d_f(\tv)=\sum_{\te\in f^{-1}(e)}d_f(\te)
\end{equation*}
for any $v\in V(G)$ or any $e\in E(G)$. 

The \emph{ramification degree} of a harmonic morphism $f:\tG\to G$ of weighted graphs at a vertex $\tv\in V(\tG)$ is 
\begin{equation*}
\label{eq:localRH}
\Ram_f(\tv)=d_f(\tv)\chi(f(\tv))-\chi(\tv).
\end{equation*}
A harmonic morphism $f:\tG\to G$ is called \emph{effective} if $\Ram_f(\tv)\geq 0$ for all $\tv\in V(\tG)$ and \emph{unramified} if $\Ram_f(\tv)=0$ for all $\tv\in V(\tG)$. 

A harmonic morphism of metric graphs $\phi:\tGa\to \Ga$ is a harmonic morphism of the underlying graphs $f:\tG\to G$ with the additional constraint that
\begin{equation}
\ell(f(\te))=d_f(\te)\ell(\te)
\label{eq:dilation}
\end{equation}
for any edge $\te\in E(\tG)$, so that $\phi$ represents a dilation by an integer factor $d_f(\te)$ along $\te$.

\subsection{Edge contractions} Let $(G,g)$ be a weighted graph and $e\in E(G)$ an edge. We define the \emph{contraction} $G_e$ of $G$ along $e$ by merging its root vertices and adding their genera (if $e$ is not a loop) or removing $e$ and increasing the genus of the root vertex by one (if $e$ is a loop). We define the \emph{contraction} $G_F$ of $G$ along a set of edges $F\subset E(G)$ iteratively. It is elementary to verify that edge contraction preserves genus. 

We similarly define the \emph{contraction} $f_F:\tG_{\tF}\to G_F$ of a harmonic morphism $f:\tG\to G$ along a set of edges $F\subset E(G)$. First, we contract $\tG$ along the preimage edges $\tF=f^{-1}(F)$ and define the map $f_F:\tG_{\tF}\to G_F$ in the obvious manner. Second, for each connected component $\tG_i$ of $\tG$ that contracts to a vertex $\tv_i\in \tG_{\tF}$, we set the local degree to be $d_{f_{F}}(\tv_i)=\deg f|_{\tG_i}$, the global degree of $f$ on the contracted subgraph. It is elementary to verify that $f_F$ is also harmonic and $\deg f_F=\deg f$, and furthermore that if $f$ is effective or unramified, then so is $f_F$.

\subsection{Double covers} A \emph{double cover} $p:\tG\to G$ of weighted graphs is an unramified harmonic morphism of degree two. A vertex $v\in V(G)$ is called \emph{free} if $p^{-1}(v)=\{\tv^+,\tv^-\}$ with $d_p(\tv^+)=d_p(\tv^-)=1$ and \emph{dilated} if $p^{-1}(v)=\{\tv\}$ with $d_p(\tv)=2$, and we similarly define free and dilated edges. By the local degree condition, the root vertices of a dilated edge are dilated, so the dilated vertices and edges form a subgraph called the \emph{dilation subgraph}. We say that $p$ is \emph{free} if the dilation subgraph is empty, and \emph{edge-free} if the dilation subgraph consists of isolated vertices. The \emph{torus rank} of a double cover $\pi:\tG\to G$ is
\[
t(\tG/G)=b_1(\tG)-b_1(G).
\]

\begin{example} The \emph{Friedman--Smith} double cover $\pi:\tG\to G$, denoted $FS_n$, is the torus rank $n$ double cover shown on Figure~\ref{fig:FSn}. Both vertices of $G$ are dilated, and no edges are dilated. 
\begin{figure}
    \centering
    \begin{tikzpicture}

    \draw[fill] (0,0) circle(0.07);
    \draw[fill] (3,0) circle(0.07);
    \draw[thin] (0,0) .. controls (1,0.5) and (2,0.5) .. (3,0);
    \draw[thin] (0,0) .. controls (1,-0.5) and (2,-0.5) .. (3,0);
    \draw[fill] (1.5,0) circle (0.02);
    \draw[fill] (1.5,0.2) circle (0.02);
    \draw[fill] (1.5,-0.2) circle (0.02);
    \node at (1.5,0.6) {$e_1$};
    \node at (1.5,-0.6) {$e_n$};

\begin{scope}[xshift=0cm,yshift=3cm]

    \draw[fill] (0,0) circle(0.1);
    \draw[fill] (3,0) circle(0.1);
    \draw[thin] (0,0) .. controls (1,1.3) and (2,1.3) .. (3,0);
    \draw[thin] (0,0) .. controls (1,-1.3) and (2,-1.3) .. (3,0);
    \draw[thin] (0,0) .. controls (1,0.6) and (2,0.6) .. (3,0);
    \draw[thin] (0,0) .. controls (1,-0.6) and (2,-0.6) .. (3,0);
    \draw[fill] (1.5,0) circle (0.02);
    \draw[fill] (1.5,0.2) circle (0.02);
    \draw[fill] (1.5,-0.2) circle (0.02);
    \node at (1.5,1.3) {$\te^+_1$};
    \node at (1.5,-0.7) {$\te^+_n$};
    \node at (1.5,0.7) {$\te^-_1$};
    \node at (1.5,-1.3) {$\te^-_n$};
\end{scope}
    \end{tikzpicture}
    \caption{The $FS_n$ double cover}
    \label{fig:FSn}
\end{figure}
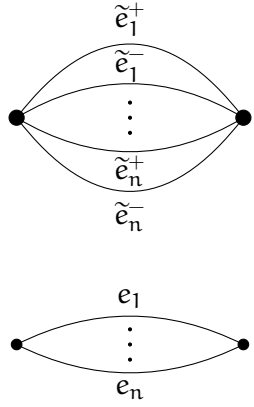
    
\end{example}

A free double cover $p:\tG\to G$ is a degree two cover in the topological sense and is equivalent to the structure of a \emph{signed graph} on $G$, which is a choice of sign on each edge of $G$ up to vertex switching equivalence. In terms of covering space theory, a free double cover is a Galois cover with group $\ZZ/2\ZZ$ and a signed graph structure is an element of the cohomology group $H^1(G,\ZZ/2\ZZ)$. The trivial element corresponds to the trivial double cover, in particular any free double cover of a tree is trivial. 

The contraction of an unramified morphism is unramified, hence the contraction $p_F$ of a double cover $p:\tG\to G$ along a set of edges $F\subset E(G)$ is a double cover. We note that $p_F$ may be dilated even if $p$ is free, this happens if and only if $F$ contains an odd cycle, that is to say, a simple cycle such that the product of the signs along its edges is negative.

Conversely, given an edge-free double cover $p:\tG\to G$, we can (non-canonically) construct a free double cover $p_0:\tG_0\to G_0$ which contracts to $p$ as follows. First, it is elementary to check that an isolated dilated vertex has positive genus. Hence we can attach a loop $e_v$ at every dilated vertex $v\in V(G)$ (reducing the genus by one), and replace the preimage vertex $\tv=p^{-1}(v)$ with a pair of vertices $\tv^+,\tv^-$ mapping to $v$, connected by a pair of edges mapping to $e_v$.

\subsection{Graph gonality}\label{subsec:gonality} The algebraic notion of gonality has several different generalizations for metric graphs, the best-known of which is perhaps divisorial Baker--Norine gonality~\cite{baker2007riemann}. We use a stronger definition, advanced in the paper paper~\cite{2018CoolsDraisma}, in terms of the existence of a harmonic map to a metric tree. This definition is better behaved in moduli, and is also the natural setting for the tropical trigonal construction, which is the main technical tool of this paper. 

\begin{definition} Let $\Ga$ be a metric graph. A \emph{tropical modification} of $\Ga$ is a metric graph $\Ga'$ obtained by attaching finitely many metric trees to $\Ga$. An \emph{$n$-gonal structure} on $\Ga$ is an effective harmonic morphism $\ph:\Ga'\to \De$ of degree $n$, where $\Ga'$ is a tropical modification of $\Ga$ and $\De$ is a metric tree of genus zero. A metric graph $\Ga$ is called \emph{$n$-gonal} if it admits an $n$-gonal structure.
\end{definition}

It was shown in~\cite{2018CoolsDraisma} that this definition of gonality satisfies the same properties, in terms of dimensions in moduli, as algebraic gonality:

\begin{theorem}[Theorem 1 and Corollary 14 in~\cite{2018CoolsDraisma}, see also Proposition 2.6 in~\cite{2025Zakharov}] Let $\Ga$ be a metric graph and let $n\geq \lceil(g+2)/2\rceil$, where $g$ is the genus of $\Ga$. Then there exists an $n$-gonal structure $\ph:\Ga'\to \De$ on $\Ga$, such that furthermore the target tree $\Delta$ has at most $2g+2n-5$ edges.
\label{thm:CD}

\end{theorem}

\subsection{Tropical principally polarized abelian varieties} A \emph{tropical ppav} $X=(\Lambda,[\cdot,\cdot])$ is determined by a free abelian group $\Lambda$ of rank $g$ and a symmetric bilinear pairing $[\cdot,\cdot]:\Lambda\times \Lambda\to \RR$ that induces a positive definite quadratic form on $\Lambda_{\RR}$. There is an inclusion $\Lambda\subset \Lambda^*_{\RR}=\Hom(\Lambda,\RR)$ determined by the map $\lambda\mapsto [\cdot,\lambda]$, and the ppav itself is the $g$-dimensional torus $\Hom(\Lambda,\RR)/\Lambda$. The \emph{Voronoi polytope} of $X$ is the convex symmetric polyhedral domain
\[
\Vor(X)=\{x\in \Lambda^*_{\RR}:\|x\|\leq \|x-\lambda\|\mbox{ for all }\lambda\in \Lambda\}.
\]

We consider two examples of tropical ppavs. Let $\Ga$ be a metric graph with model $G$. The \emph{tropical Jacobian} $\Jac(\Ga)$ is the dimension $g(G)$ tppav determined by the lattice $\Lambda=H_1(G,\ZZ)$ together with the \emph{integration pairing}\footnote{The first copy of $H_1(\Ga,\ZZ)$ should be viewed as the space of \emph{harmonic 1-forms} on $\Ga$, so that the pairing is the integral of a $1$-form along a 1-cycle.} 
\begin{equation}
H_1(G,\ZZ)\times H_1(G,\ZZ)\to \RR,\quad
\left[\sum a_ee,\sum b_ee\right]=\sum a_eb_e\ell(e).
\label{eq:integrationpairing}
\end{equation}
If $\Ga$ has vertex weights, they do not contribute to the Jacobian. Tropical Abel--Jacobi theory identifies $\Jac(\Ga)$ with $\Pic_0(\Ga)$, the set of equivalence classes of degree zero divisors on $\Ga$.

The definition of the \emph{tropical Prym variety} of a double cover of metric graphs $\pi:\tGa\to \Ga$ is somewhat more subtle. The original definition, given by Jensen and Len in~\cite{2018JensenLen}, is based on the algebraic divisorial definition. The double cover defines a norm map 
\[
\pi_*:\Jac(\tGa)\to \Jac(\Ga)
\]
on the tropical Jacobians, and the \emph{divisorial Prym variety} $\Prym_d(\tGa/\Ga)$ is the connected component of the identity of $\Ker \pi_*$. This object suffers from a number of related problems: it does not carry a natural principal polarization, it does not satisfy the expected universal property, and, most importantly for us, it behaves poorly in moduli. Specifically, for a double cover $\pi_F:\tGa_{\tF}\to \Ga_F$ obtained from $\pi:\tGa\to \Ga$ by contracting a set of edges $F\subset E(\Ga)$, the divisorial Prym variety $\Prym_d(\tGa_F/\Ga_F)$ is not in general equal to the limit of $\Prym_d(\tGa/\Ga)$. This problem was resolved in the paper~\cite{2025RoehrleZakharovA}, where we constructed an alternative object, the \emph{continuous Prym variety} $\Prym_c(\tGa/\Ga)$ of a double cover $\pi:\tGa\to \Ga$. Incidentally, this object turns out to be the correct one for studying the extension problem for the Prym period map.

The double cover $\pi:\tGa\to \Ga$ determines an involution $\iota:\tGa\to \tGa$ that exchanges the two points in every fiber over a free point and fixes the dilated points. Let $\iota_*:H_1(\tGa,\ZZ)\to H_1(\tGa,\ZZ)$ be the induced map on the homology and consider the lattice $\Lambda=\Im(\Id-\iota_*)\subset H_1(\tGa,\ZZ)$ of rank $t(\tGa/\Ga)$.

\begin{definition}The \emph{continuous Prym variety} $\Prym_c(\tGa/\Ga)$ of the double cover $\pi:\tGa\to \Ga$ is the dimension $t(\tGa/\Ga)$ tropical ppav defined by the lattice $\Lambda$ and the bilinear pairing $[\cdot,\cdot]:\Lambda\times \Lambda\to \RR$ that is $\frac{1}{2}$ of the restriction of the integration pairing on $H_1(\tG,\ZZ)$. 
\label{def:Prymc}
\end{definition}

The exact relationship between $\Prym_c(\tGa/\Ga)$ and $\Prym_d(\tGa/\Ga)$ is described in~\cite{2025RoehrleZakharovA} and depends on the number of connected components of the dilation subgraph. For our purposes, it is enough to note that
\[
\Prym_c(\tGa/\Ga)=\Prym_d(\tGa/\Ga)
\]
for any free double cover $\tGa\to \Ga$ of metric graphs (also for any double cover with connected dilation subgraph).

\subsection{The trigonal construction} Finally, we recall the tropical version~\cite{2025RoehrleZakharovA} of the algebraic trigonal construction of Recillas~\cite{1974Recillas}. Let $\pi:\tGa\to \Ga$ be a free double cover of a trigonal tropical curve $\ph:\Ga\to \De$ (if it is necessary to attach trees to $\Ga$ to define $\ph$, we attach two copies of each tree to $\tGa$ and extend $\pi$ trivially). For $x\in \De$, define the pullback
\[
\ph^*(x)=\sum_{y\in \ph^{-1}(x)}d_{\ph}(y)y\in \Div(\Ga),
\]
where $\Div(\Ga)$ is the free abelian group on the points of $\Ga$. Similarly, we define the pushforward of a divisor on $\tGa$:
\[
\pi_*\sum_{\widetilde{y}\in \tGa} a_{\widetilde{y}}\widetilde{y}=\sum_{\widetilde{y}\in \tGa} a_{\widetilde{y}}\pi(\widetilde{y}).
\]
We now consider the set of sections of the fibers of $\pi:\tGa\to \Ga$ over the various points of $\Delta$:
\[
\widetilde{\Pi}=\left\{D\in \Div(\tGa):D\geq 0\mbox{ and }\pi_*(D)=\ph^*(x)\mbox{ for some (necessarily unique) }x\in \Delta\right\}.
\]

\begin{theorem}[Theorem 5.1 in~\cite{2025RoehrleZakharovA}] Let $\pi:\tGa\to \Ga$ be a connected free double cover of metric graphs and let $\ph:\Ga\to \De$ be a harmonic morphism of degree 3, where $\De$ is a metric tree. \label{thm:trigonal} 

\begin{enumerate}
    \item There is a natural structure of a metric graph on $\widetilde{\Pi}$ such that the map $\widetilde{\psi}:\widetilde{\Pi}\to \De$, which sends $D\in \widetilde{\Pi}$ to the unique $x\in \Delta$ such that $\pi_*(D)=\ph^*(x)$, is a harmonic morphism of degree 8.

    \item The metric graph $\widetilde{\Pi}$ consists of two isomorphic connected components $\Pi$ that are exchanged by the involution that is induced by the fixed-point-free involution $\iota:\tGa\to \tGa$ corresponding to the double cover, and the restriction $\psi:\Pi\to \De$ of $\widetilde{\psi}$ to each component is a harmonic morphism of degree 4. 
    \item The metric graph $\Pi$ has genus $g(\Pi)=g(\Ga)-1$, and the continuous Prym variety of the free double cover $\pi:\tGa\to \Ga$ and the Jacobian variety of $\Pi$ are isomorphic as tropical ppavs:
\[
\Prym_c(\tGa/\Ga)\simeq \Jac(\Pi).
\]

\end{enumerate}

\end{theorem}

We note that this theorem is stated in~\cite{2025RoehrleZakharovA} for the divisorial Prym variety, which agrees with the continuous Prym variety for free double covers.

\section{Toroidal compactifications and extensions of period maps}

In this section, we briefly review toroidal compactifications of $A_g$ and the extension problem for the Torelli and Prym maps, and rephrase the latter in terms of tropical geometry. We then construct an explicit minimal resolution of the Prym period map $\overline{p}_4:\overline{R}_4\dashrightarrow \overline{A}_3$, which is our main result.

\subsection{Admissible decompositions} The theory of toroidal compactifications of $A_g$ was developed over $\CC$ in~\cite{1975AshMumfordRapoportTai} and extended to the arithmetic setting in~\cite{1990FaltingsChai}, and is parallel to the theory of toric varieties. A toroidal compactification $\overline{A}_g^{\tau}$ is determined by a combinatorial gadget $\tau$ called an \emph{admissible decomposition}, and we recall the necessary definitions.

Fix a lattice $\Lambda=\ZZ^g$. The space $(\Sym^2\Lambda)\otimes_{\ZZ} \RR$ of quadratic forms on $\Lambda$ contains the open cone $\Omega_g$ of positive definite quadratic forms. Denote by $\Omega_g^{\rt}$ the \emph{rational closure} of $\Omega_g$, consisting of positive semidefinite quadratic forms whose null space is defined over $\QQ$. The group $\GL(\Lambda)$ naturally acts on $\Omega_g^{\rt}$ by change of basis.

\begin{definition} An \emph{admissible decomposition} of $\Omega_g^{\rt}$ is a rational polyhedral fan $\tau$ satisfying the following conditions:
\begin{enumerate}

\item $\Supp \tau=\Omega_g^{\rt}$.
\item $\tau$ is preserved by the $\GL(\Lambda)$-action.
\item There are finitely many orbits of cones under the $\GL(\Lambda)$-action.

\end{enumerate}
    
\end{definition}

There are three standard admissible decompositions, each defining a corresponding toroidal compactification of $A_g$. The three fans coincide for $g=3$, the case that we are interested in, and we denote the common toroidal compactification by $\overline{A}_3$. We recall the definition of one of the decompositions. Recall that the \emph{Voronoi polytope} $\Vor(Q)$ of a positive definite quadratic form $Q\in \Omega_g$ is the set of points closer to the origin than to any other lattice point:
\[
\Vor(Q)=\{x\in \RR^g:Q(x)\leq Q(x-v)\mbox{ for all }v\in \ZZ^g\}.
\]
The Voronoi polytope is a convex symmetric polytope cut out by a collection of hyperplanes, corresponding to those nonzero points  $v\in\ZZ^g$ for which the equality $Q(x)=Q(x-v)$ is achieved for some $x\in \RR^g$. These hyperplanes determine the normal fan of $\Vor(Q)$, which is constant near a generic $Q$ but exhibits wall-crossing behavior. Recording how the normal fan changes as we vary $Q$, we obtain the \emph{second Voronoi decomposition}
\[
\Omega_{g}=\bigcup_{C\in \tau^{V}} C,
\]
where two quadratic forms $Q_1$ and $Q_2$ lie in the same cone if and only if $\Vor(Q_1)$ and $\Vor(Q_2)$ are normally equivalent, that is to say, have the same normal fan.

We now review the extension problem for the Torelli $\overline{t}^{\tau}_g:\overline{M}_g\dashrightarrow \overline{A}^{\tau}_g$ and Prym $\overline{p}^{\tau}_g:\overline{R}_g\dashrightarrow \overline{A}^{\tau}_{g-1}$ period maps, following~\cite{2017Casalaina-MartinGrushevskyHulekLaza}. 

\subsection{The Torelli map and tropical Jacobians} We first consider the Torelli map. The moduli space of stable curves $\overline{M}_g$ comes with a stratification
\[
\overline{M}_g=\bigsqcup_{G} M(G)
\]
indexed by weighted stable graphs of genus $g$, where $M(G)$ is the relatively open stratum of stable curves with dual graph $G$ (the vertices and edges correspond to the irreducible components and the nodes, respectively). A stratum $M(G_1)$ lies in the closure of $M(G_2)$ if and only if $G_2$ is an edge contraction of $G_1$. Fix a curve $X\in M(G)$ and assume for simplicity that $G$ has trivial vertex genera. Locally near $X$, the codimension one boundary strata of $\overline{M}_g$ are a collection of simple normal crossing hyperplanes and are indexed by the edges of $G$ (the corresponding graphs are obtained by contracting all other edges).

Choose a basis for the lattice $H_1(G,\ZZ)\simeq \ZZ^g$ (if $G$ has vertex genera, we fix a surjection $\ZZ^g\twoheadrightarrow H_1(G,\ZZ)$). Each edge $e\in E(G)$ defines a quadratic form $(e^{\vee})^2\in \Sym^2(H_1(G,\ZZ)^{\vee})$ of rank one (for example, $e^{\vee}$ may be viewed as a coedge in $H^1(G,\ZZ)=H_1(G,\ZZ)^{\vee}$). By the Picard--Lefschetz formula, $(e^{\vee})^2$ is the quadratic form obtained from the monodromy around the boundary component of $\overline{M}_g$ corresponding to the edge $e$. The \emph{monodromy cone} of $X$ (which depends only on $G$) is the cone spanned by these quadratic forms
\[
C(G)=\RR_{\geq 0} \langle (e^\vee)^2\rangle_{e\in E(G)}\subset \Omega^{\rt}_g.
\]
The Torelli map $\overline{t}_g^{\tau}$ extends to $X$ (and therefore all of $M(G)$) if and only if $C(G)$ is contained in a cone of the admissible decomposition $\tau$. Mumford and Namikawa~\cite{1976Namikawa} and Alexeev~\cite{2012AlexeevBrunyate} proved this for the second Voronoi and perfect cone decompositions, respectively. 

We now make the following elementary observation. A point $Q=\sum \ell(e)(e^{\vee})^2$ in the interior of the monodromy cone is determined by choosing positive real numbers $\ell(e)$ for each $e\in E(G)$. Interpreting these as the edge lengths of a metric graph $\Ga$ with model $G$, we see that $Q$ is the quadratic form defining the tropical Jacobian $\Jac(\Ga)$. In other words, we may view the cone $\RR_{\geq 0}^{E(G)}$ as a stratum in the moduli space $M_g^{\trop}$ of tropical curves, parametrizing metric graphs with model $G$ (and their edge contractions), and the monodromy cone $C(G)$ as its image in the moduli space $A_g^{\trop}=\Omega_g^{\rt}/\GL(g,\ZZ)$ of tropical ppavs under the tropical Torelli map.

Using the criterion for second Voronoi cones described above, we immediately obtain the following result, which is simply a restatement of the Mumford--Namikawa extension theorem for $\overline{t}_g^V$ in tropical terms.                                                      

\begin{proposition} Let $G$ be a graph. The Voronoi polytopes $\Vor(\Jac(\Ga))$ of the tropical Jacobians of all metric graphs $\Ga$ with model $G$ are normally equivalent. \label{prop:Torelliextension}

\end{proposition}

\subsection{The Prym map and tropical Pryms} The moduli space $\overline{R}_g$ of admissible double covers of genus $g$ curves (see~\cite{1977Beauville}) likewise comes with a stratification
\[
\overline{R}_g=\bigsqcup_{\tG\to G}R(\tG/G)
\]
indexed by double covers of weighted graphs of genus $g$, where $R(\tG/G)$ parametrizes admissible double covers with dual cover $\tG\to G$. Similarly to $\overline{M}_g$, the codimension of $R(\tG/G)$ is equal to the number of edges of $G$, and an edge contraction of $\tG\to G$ to $\tG_{\tF}\to G_F$ corresponds to an inclusion of $R(\tG/G)$ in the closure of $R(\tG_{\tF}/G_F)$. Locally near a point $(\tX\to X)\in R(\tG/G)$, the boundary of $\overline{R}_g$ is a collection of simple normal crossing hyperplanes indexed by the edges of $E(G)$. The monodromy around these boundary components was computed in~\cite{2017Casalaina-MartinGrushevskyHulekLaza}, and we recall the description.

The double cover $p:\tG\to G$ determines an involution $\iota:\tG\to \tG$, and we consider the lattice $\Lambda=\Im (\Id-\iota_*)\subset H_1(\tG,\ZZ)$ of rank $t(\tG/G)\leq g-1$. For each edge $e\in E(G)$ we pick a preimage edge $\te\in E(\tG)$ and consider the rank one quadratic form $(\te^{\vee}-\iota\te^{\vee})^2\in \Sym^2(\Lambda^{\vee})$ (note that if $e$ is dilated, then $\iota\te=\te$ and the quadratic form is zero). By Proposition 4.3 in~\cite{2017Casalaina-MartinGrushevskyHulekLaza}, $(\te^{\vee}-\iota\te^{\vee})^2$ is the quadratic form obtained by monodromy around the boundary component of $\overline{R}_g$ corresponding to the edge $e$. Hence the \emph{monodromy cone} $C(\tG/G)$ of $\tX\to X$, which depends only on the dual double cover $\tG\to G$, is the cone spanned by these quadratic forms:
\begin{equation}
\phi:\RR_{\geq 0}^{E(G)}\to C(\tG/G)\subset \Omega_{g-1}^{\rt},\quad \sum \ell(e) e\mapsto \sum \ell(e)(\te^{\vee}-\iota\te^{\vee})^2,
\label{eq:tropicalPrymperiod}
\end{equation}
where we embed into $\Omega_{g-1}^{\rt}$ by choosing a surjection $\ZZ^{g-1}\twoheadrightarrow \Lambda$. 

We now fix an admissible decomposition $\tau$ of $\Omega^{\rt}_{g-1}$. The Prym period map $\overline{p}_g^{\tau}:\overline{R}_g\dashrightarrow\overline{A}_{g-1}^{\tau}$ extends over the stratum $R(\tG/G)$ if and only if $C(\tG/G)$ is contained in a cone of the admissible decomposition $\tau$. It was shown in~\cite{2002AlexeevBirkenhakeHulek} and~\cite{2002Vologodsky} that this holds for the second Voronoi decomposition if and only if $\tG\to G$ does not admit an edge contraction to the $\FS_n$ double cover for any $n\geq 2$.

More generally, a rational polyhedral fan $\mathcal{F}$ decomposing the cone $\RR_{\geq 0}^{E(G)}$ determines a toric blowup of $\overline{R}_g$ along $R(\tG/G)$. This blowup resolves $\overline{p}_g^{\tau}$ along $R(\tG/G)$ if and only if the image of each cone in the fan $\mathcal{F}$ under $\phi$ is contained in a cone of $\tau$. If we choose such a fan for each $\RR_{\geq 0}^{E(G)}$ in a way that agrees on common faces, then these blowups glue together and resolve the Prym period map on all of $\overline{R}_g$. In particular, for each double cover $\tG\to G$ there is a unique \emph{minimal decomposition}, obtained by pulling back the admissible decomposition $\tau$ along $\phi$.

We now rephrase the monodromy cone construction in the language of tropical geometry. Let $p:\tG\to G$ be a double cover. We view a point $\sum \ell(e)e$ in the interior of the cone $\RR_{\geq 0}^{E(G)}$ as a choice of positive edge lengths on $G$. By the dilation condition~\eqref{eq:dilation}, there is a unique way to define edge lengths on $\tG$ that upgrades $p$ to a double cover $\pi:\tGa\to \Ga$ of weighted metric graphs. It is then elementary to see that $\sum \ell(e)(\te^{\vee}-\iota\te^{\vee})^2$ is the quadratic form on the lattice $\Lambda=\Im(\Id-\iota)$ that defines the continuous Prym variety $\Prym_c(\tGa/\Ga)$\footnote{The paper~\cite{2017Casalaina-MartinGrushevskyHulekLaza} uses the lattice $H_1(\tG,\ZZ)^{[-]}=\frac{1}{2}\Lambda$, where the $\frac{1}{2}$ corresponds to the fact that the principal polarization on the continuous Prym variety is half of the induced polarization.}. Therefore, we may view $\RR_{\geq 0}^{E(G)}$ as a stratum in the moduli space of tropical admissible double covers $R_{g}^{\trop}$, parametrizing double covers with model $\tG\to G$ (and their edge contractions), and the map $\phi$ as the tropical Prym period map that assigns to a double cover $\tGa\to \Ga$ its (continuous) Prym variety $\Prym_c(\tGa/\Ga)$.

We now rephrase the extension criterion for the Prym period map to the second Voronoi compactification in tropical terms.

\begin{proposition}
    \label{prop:Prymextension} Let $\tG\to G$ be a double cover of weighted graphs of genus $g=g(G)$ and let $\mathcal{F}$ be a rational polyhedral fan decomposing the cone $\RR_{\geq 0}^{E(G)}$ of edge lengths of $G$. Then $\mathcal{F}$ resolves the Prym period map $\overline{p}^V_g:\overline{R}_g\dashrightarrow \overline{A}_{g-1}^V$ to the second Voronoi compactification if and only if for any two double covers $\tGa_1\to \Ga_1$ and $\tGa_2\to \Ga_2$ of metric graphs lying in the same cone of $\mathcal{F}$, the Voronoi polytopes $\Vor(\Prym_c(\tGa_1/\Ga_1))$ and $\Vor(\Prym_c(\tGa_2/\Ga_2))$ are normally equivalent.
\end{proposition}

\subsection{The trigonal cone decomposition and the resolution of the Prym period map} We now construct, for every double cover $\tG\to G$ of weighted graphs with $g(G)=4$, a decomposition of the edge length cone $\RR_{\geq 0}^{E(G)}$ that resolves the Prym period map $\overline{p}_4:\overline{R}_4\dashrightarrow \overline{A}_3$ along the boundary stratum $R(\tG/G)$. This resolution is not minimal, indeed, it involves blowups outside the indeterminancy locus of $\overline{p}_4$. However, the minimal resolution can be constructed by joining the resulting cones, and this is the resolution that is given in the Appendix.

We first make the following two observations. Let $p:\tG\to G$ be a free double cover of weighted graphs with $g(G)=4$, and assume that $G$ has the maximal number of $|E(G)|=3g-3=9$ edges; this happens if $G$ is trivalent and has trivial vertex genera. The faces of the cone $\RR_{\geq 0}^{E(G)}$ are in one-to-one correspondence with edge contractions of $p$, and a cone decomposition of $\RR_{\geq 0}^{E(G)}$ resolving the Prym period map at the boundary stratum $R(\tG/G)$ (which is a point) induces a resolving cone decomposition on each face of $\RR_{\geq 0}^{E(G)}$. Any edge-free double cover can be obtained as a contraction of a free double cover, which we can further assume to have the maximal number of edges. Therefore, to resolve the Prym period map on all boundary strata corresponding to edge-free double covers, it is enough to find finding resolving cone decompositions for all free double covers $\tG\to G$ with $|E(G)|=9$ edges.

Now let $p:\tG\to G$ be a dilated double cover, let $F\subset E(G)$ be the set of dilated edges, and let $p_F:\tG_{\tF}\to G_F$ be the edge-free double cover obtained by contracting each connected component of the dilation subgraph to a point. Let $\pi:\tGa\to \Ga$ be a double cover of metric graphs with model $p$ and let $\pi_F:\tGa_{\tF}\to \Ga_{F}$ be the corresponding contraction. Viewing $\pi$ and $\pi_F$ as points of the cones $\RR_{\geq 0}^{E(G)}$ and $\RR_{\geq 0}^{E(G_F)}$, respectively, the assignment $\pi\mapsto \pi_F$ is the natural projection $\RR_{\geq 0}^{E(G)}\to\RR_{\geq 0}^{E(G_F)}$. The involution $\iota:\tGa\to \tGa$ associated to the double cover $\pi$ fixes the dilated edges, therefore the latter do not show up in the lattices defining the continuous Prym variety (see Definition~\ref{def:Prymc}). Hence the continuous Prym varieties are canonically isomorphic:
\[
\Prym_c(\tGa/\Ga)=\Prym_c(\tGa_{\tF}/\Ga_F).
\]
Therefore, by Proposition~\ref{prop:Prymextension}, for any resolution of the edge length cone $\RR_{\geq 0}^{E(G_F)}$ that resolves the Prym period map $\overline{p}_4$ along $R(\tG_{\tF}/G_F)$, its pullback to $\RR_{\geq 0}^{E(G)}$ via the projection map resolves $\overline{p}_4$ along $R(\tG/G)$.

Together, these two observations imply that it suffices to find resolving cone decompositions of the cones $\RR_{\geq 0}^{E(G)}$ corresponding to free double covers $\tG\to G$, where the target graph $G$ has trivial vertex weights and $|E(G)|=9$ edges. Let $G$ be such a graph and let $\Ga$ be a metric graph with model $G$. By the main result of~\cite{2018CoolsDraisma} (restated here as Theorem~\ref{thm:CD}), there exists a tropical modification $\Ga'$ of $\Ga$ admitting a trigonal map $\varphi:\Ga'\to \De$ to a metric tree $\De$. The combinatorial type of the model $f:G'\to D$ of $\varphi$ depends, in general, on the edge lengths of $\Ga$. However, we can decompose the cone $\RR_{>0}^{E(G)}$ of edge lengths of $G$ into subcones, with the property that any two metric graphs $\Ga_1$ and $\Ga_2$ lying in the same subcone admit trigonal structures $\varphi_1:\Ga'_1\to \De$ and $\varphi_2:\Ga'_2\to \De$ with the same underlying model. Extending to the boundary, we obtain a cone decomposition
\begin{equation}
    \RR_{\geq 0}^{E(G)}=\bigcup_{\mu\in \Sigma_G} C_\mu,
\label{eq:trigonalcones}
\end{equation}
which defines the \emph{trigonal decomposition} of the cone $\RR_{\geq 0}^{E(G)}$ (the decomposition is the same for all double covers $\tG\to G$). As explained above, we then pass to faces to obtain cone decompositions for all edge-free double covers, and then trivially induce cone decompositions for all dilated double covers. 

The exact manner in which the trigonal cones are enumerated is outlined below. The attentive reader may note that the trigonal decomposition of a cone $\RR_{\geq 0}^{E(G)}$ is not canonical (this is due due to the fact that a generic metric graph of genus 4 admits \emph{two} trigonal structures), so the decompositions on two maximal cones may not agree on their common face. This issue can be fixed by passing to a common refinement. However, these considerations are ultimately not important, because the trigonal decomposition plays an intermediate role and because the minimal cone decomposition described in the Appendix does have the required consistency on faces.

We now prove the main technical result of this paper.

\begin{theorem} The trigonal decomposition resolves the Prym period map $\overline{t}_4:\overline{R}_4\dashrightarrow \overline{A}_3$.
    
\end{theorem}

\begin{proof} By the observations above, it is sufficient to prove the theorem for all free double covers $p:\tG\to G$ with $|E(G)|=9$. Let $\pi:\tGa\to \Ga$ be a free double cover of metric graphs with model $p$. By Theorem~\ref{thm:CD}, there exists a trigonal structure $\varphi:\Ga'\to \De$, where $\Ga'$ is a tropical modification of $\Ga$. For each tree attached to $\Ga'$, we attach two identical trees to $\tGa$ and extend $\pi$ to a double cover $\pi':\tGa'\to \Ga'$. By the trigonal construction (see Theorem~\ref{thm:trigonal}), the tower $\tGa'\to \Ga'\to \De$ determines a tetragonal structure $\kappa:\Pi\to \De$, and furthermore
\[
\Prym_c(\tGa'/\Ga')\simeq \Prym_c(\tGa/\Ga)\simeq \Jac(\Pi).
\]

We now consider two double covers $\pi_1:\tGa_1\to \Ga_1$ and $\pi_2:\tGa_2\to \Ga_2$ lying in the same cone $C_{\mu}$ of the trigonal decomposition. By construction, the trigonal structures $\varphi_1:\Ga_1'\to \De_1$ and $\varphi_2:\Ga_2'\to \De_2$ have the same underlying model, therefore the same is true of the corresponding tetragonal structures $\kappa_1:\Pi_1\to \De_1$ and $\kappa_2:\Pi_2\to \De_2$. In particular, $\Pi_1$ and $\Pi_2$ have the same underlying graph, so by Proposition~\ref{prop:Torelliextension}, the Voronoi polytopes of $\Jac(\Pi_1)$ and $\Jac(\Pi_2)$ are normally equivalent. But then so are the Voronoi polytopes of $\Prym_c(\tGa_1/\Ga_1)$ and $\Prym_c(\tGa_2/\Ga_2)$, so by Proposition~\ref{prop:Prymextension} the trigonal decomposition resolves the Prym period map along $R(\tG/G)$.

\end{proof}

We now briefly describe how the cones $\{C_\mu\}$ are constructed (see~\cite{2025Zakharov} for a detailed exposition). Let $G$ be a stable graph of genus 4 with the maximal number of $3g-3=9$ edges, and let $\Ga\in \RR_{>0}^{E(G)}$ be a metric graph with model $G$. By Theorem~\ref{thm:CD}, there exists a trigonal structure $\Ga'\to \De$ on a tropical modification of $\Ga$ such that the target tree $\De$ has at most $2g+2n-5=9$ edges. For a generic $\Ga$, $\De$ has the full set of 9 edges. Furthermore, $\Delta$ has no vertices of valence greater that 3, otherwise $\Ga$ is not trivalent.  

There are 37 trees on 9 edges having no vertices of valence greater that 3. For each such tree $D$, we carefully enumerate all degree 3 effective harmonic morphisms $f:G'\to D$ by choosing appropriate markings on the vertices and edges of $D$. Of these, we retain only those for which the stabilization $G$ of $G'$ is a trivalent graph of genus 4. The edge lengths of $G$ are linear combinations of edge lengths of $D$, and we discard all trigonal structures for which the edge lengths of $G$ are not linearly independent, since the resulting cones are not of maximal dimension.

The result is a list of 821 trigonal structures $G'\to D$. For each such trigonal structure, varying the edge lengths of $D$ gives a subcone of the cone $\RR_{\geq 0}^{E(G)}$ of edge lengths. The union of the closures of these subcones cover $\RR_{\geq 0}^{E(G)}$, and this is the trigonal decomposition. The entire construction was coded manually in Sage and runs on a 2020 MacBook Air in under an hour.

\subsection{The minimal resolution} It is immediately evident that the trigonal decomposition does not define the minimal resolution of the Prym period map $\overline{p}_4:\overline{R}_4\dashrightarrow \overline{A}_3$. For example, a free double cover $\tG\to G$ with a single negative edge does not edge contract to any $\FS_n$ double cover, therefore the stratum $R(\tG/G)$ lies outside the indeterminancy locus and the cone $\RR_{\geq 0}^{E(G)}$ needs no decomposition. However, the minimal decomposition of a cone $\RR_{\geq 0}^{E(G)}$ required to resolve $\overline{p}_4$ along $R(\tG/G)$ can be determined in the following simple manner. Choose a point $\tGa_{\mu}\to \Ga_{\mu}$ in each trigonal cone $C_{\mu}\subset \RR_{\geq 0}^{E(G)}$ (for example, by setting all edge lengths of the target tree $\Delta$ to 1) and compute the Voronoi polytope $\Vor(\Prym_c(\tGa_{\mu}/\Ga_{\mu}))$\footnote{This portion of the code used an edited ChatGPT output.}. We then join the cones $C_{\mu}$ into larger cones according to whether or not the Voronoi polytopes $\Vor(\Prym_c(\tGa_{\mu}/\Ga_{\mu}))$ are normally equivalent. By Proposition~\ref{prop:Prymextension}, this gives the minimal decomposition resolving the Prym period map.

The minimal decomposition verifies Conjecture 5.1 in~\cite{2025Zakharov}, which I now briefly explain. The \emph{second moment} of a tropical ppav $X=(\Lambda,[\cdot,\cdot])$ is the integral
\[
I_2(X)=\int_{\Vor(X)}\|x\|^2.
\]
The second moment has arithmetic significance~\cite{2022deJongShokrieh}, and it is interesting to find explicit formulas for specific families of tropical ppavs. The second moment of a tropical Jacobian $\Jac(\Ga)$ was computed in~\cite{2023deJongShokrieh}, and the goal of~\cite{2025Zakharov} was to compute the second moment of the tropical Prym variety $\Prym(\tGa/\Ga)$ for $g(\Ga)\leq 4$ using the tropical trigonal construction. As an intermediate result, I proved that, up to a scaling factor, $I_2(\Jac(\Ga))$ is a polynomial in the edge lengths of $\Ga$. However, $I_2(\Prym(\tGa/\Ga))$ turns out to be only piecewise-polynomial in the edge lengths of $\Ga$. Morally, this is explained by Propositions~\ref{prop:Torelliextension} and~\ref{prop:Prymextension}: for $\Jac(\Ga)$, the domain of integration $\Vor(\Jac(\Ga))$ remains normally constant for all edge lengths of $\Ga$, while $\Vor(\Prym(\tGa/\Ga))$ changes polyhedral type. Conjecture 5.1 in~\cite{2025Zakharov} states that the domains of polynomiality of $I_2(\Prym(\tGa/\Ga))$ (which I found computationally) are exactly the domains where $\Vor(\Prym(\tGa/\Ga))$ has constant normal cone, and the present paper verifies this conjecture for $g=4$.

\section*{Appendix: the minimal resolution}

We now list the minimal decompositions of the edge length cones $\RR_{\geq 0}^{E(G)}$ that resolve the Prym map $\overline{p}_4:\overline{R}_4\dashrightarrow \overline{A}_3$ along the boundary stratum $R(\tG/G)$. As explained above, it is sufficient to describe these decompositions for free double covers $\tG\to G$, where $G$ is a trivalent graph having trivial vertex genera and hence $|E(G)|=9$ edges. Unlike the trigonal decomposition, the minimal decomposition of $\RR_{\geq 0}^{E(G)}$ depends on the double cover $\tG\to G$ and not just $G$.

By~\cite{2002AlexeevBirkenhakeHulek} and~\cite{2002Vologodsky}, the indeterminancy locus of $\overline{p}_4:\overline{R}_4\dashrightarrow \overline{A}_3$ is $\overline{\FS_2}\cup\overline{\FS_3}$. Therefore, the cone $\RR_{\geq 0}^{E(G)}$ of a double cover $\tG\to G$ requires a nontrivial decomposition if and only if the double cover contracts to the $\FS_2$ or an $\FS_3$ double cover. It turns out that the required decomposition depends entirely on the number and relative arrangement of such contractions.

Let $p:\tG\to G$ be a double cover. We say that a subgraph $F\subset G$ is \emph{unbalanced} if $p^{-1}(F)$ is connected, in other words, if the restriction of $p$ to $F$ is not the split free double cover. Any subgraph containing a dilated vertex is unbalanced. If $p|_{p^{-1}(F)}$ is free, then $F$ is unbalanced if and only if it has an odd cycle, so in particular a tree is balanced if it contains no dilated vertices. 

We call a set $\{e_1,\ldots,e_n\}\subset E(G)$ of free edges an \emph{$\FS_n$-bond} if the contraction of $\tG\to G$ along the complementary edges is the $\FS_n$ double cover, in other words if the following conditions hold:
\begin{enumerate}
    \item $G\backslash\{e_1,\ldots,e_n\}$ consists of two unbalanced connected components $G_1$ and $G_2$.
    \item $G\backslash F$ is connected for every proper subset $F\subset\{e_1,\ldots,e_n\}$.
\end{enumerate}
Since $g(G_1)+g(G_2)=g(G)-n+1$, $G$ may have an $\FS_n$-bond only for $n\leq g-1$. 

We now list the minimal decomposition of the edge length cone $\RR_{\geq 0}^{E(G)}$ for all free double covers $\tG\to G$ with $|E(G)|=9$ by increasing number of $\FS_2$- and $\FS_3$-bonds, which occur in a variety of configurations. By abuse of notation, we denote the coordinates on the edge length cone (in other words, the lengths of the edges) by the same letters as the edges themselves. All indicated vertices are free.

\subsubsection*{One $\FS_2$-bond, no $\FS_3$-bonds} Denote by $\{e,f\}\subset E(G)$ the unique $\FS_2$-bond, then $G\backslash\{e,f\}=G_1\sqcup G_2$, where $G_1$ and $G_2$ are unbalanced graphs with $g(G_1)=1$ and $g(G_2)=2$:
\begin{center}
\begin{tikzpicture}
    \draw[thin] (0,0) .. controls (1,0.5) and (2,0.5) .. (3,0);
    \draw[thin] (0,0) .. controls (1,-0.5) and (2,-0.5) .. (3,0);
    \node at (1.5,0.6) {$e$};
    \node at (1.5,-0.6) {$f$};
    \draw[fill,white] (0,0) circle(0.5);
    \draw[fill,white] (3,0) circle(0.5);
    \draw[thin] (0,0) circle(0.5);
    \draw[thin] (3,0) circle(0.5);
    \node at (0,0) {$G_1$};
    \node at (3,0) {$G_2$};
\end{tikzpicture}
\end{center}
There are two cones in the decomposition:
\[
\{e\leq f\},\quad \{f\leq e\}.
\]

\subsubsection*{Two $\FS_2$-bonds, no $\FS_3$-bonds} There are two subcases consider. The $\FS_2$-bonds $\{e_1,f_1\}$ and $\{e_2,f_2\}$ may be disjoint, in which case $G$ has the following form:
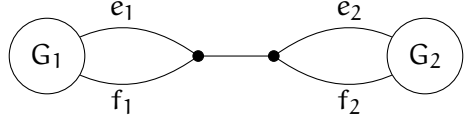
\begin{center}
\begin{tikzpicture}
    \draw[thin] (0,0) .. controls (0.7,0.5) and (1.3,0.5) .. (2,0);
    \draw[thin] (0,0) .. controls (0.7,-0.5) and (1.3,-0.5) .. (2,0);
    \node at (1,0.6) {$e_1$};
    \node at (1,-0.6) {$f_1$};
    \draw[fill,white] (0,0) circle(0.5);
    \draw[thin] (0,0) circle(0.5);
    \node at (0,0) {$G_1$};
    \draw[fill] (2,0) circle(0.07);
    \draw[thin] (2,0) -- (3,0);
    \draw[fill] (3,0) circle(0.07);
    \draw[thin] (3,0) .. controls (3.7,0.5) and (4.3,0.5) .. (5,0);
    \draw[thin] (3,0) .. controls (3.7,-0.5) and (4.3,-0.5) .. (5,0);
    \node at (4,0.6) {$e_2$};
    \node at (4,-0.6) {$f_2$};
    \draw[fill,white] (5,0) circle(0.5);
    \draw[thin] (5,0) circle(0.5);
    \node at (5,0) {$G_2$};
\end{tikzpicture}
\end{center}
Here $G_1$ and $G_2$ are unbalanced graphs of genus one. There are four cones in the decomposition:
\[
\{e_1\leq f_1,e_2\leq f_2\},\quad \{e_1\leq f_1,f_2\leq e_2\},\quad \{f_1\leq e_1,e_2\leq f_2\},\quad 
\{f_1\leq e_1,f_2\leq e_2\}.
\]
Alternatively, the $\FS_2$-bonds $\{e,f\}$ and $\{e,g\}$ may share a edge, in which case $G$ has the following form:
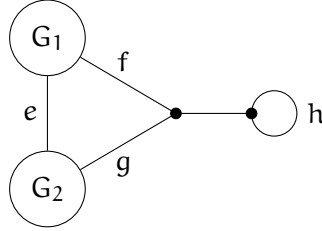
\begin{center}
\begin{tikzpicture}
    \draw[thin] (0,1) -- (0,-1) -- (1.7,0) -- (0,1);
    \draw[fill,white] (0,1) circle(0.5);
    \draw[thin] (0,1) circle(0.5);
    \node at (0,1) {$G_1$};
    \draw[fill,white] (0,-1) circle(0.5);
    \draw[thin] (0,-1) circle(0.5);
    \node at (0,-1) {$G_2$};
    \node[left] at (0,0) {$e$};
    \node at (1,0.7) {$f$};
    \node at (1,-0.7) {$g$};
    \draw[fill] (1.7,0) circle(0.07);
    \draw[thin] (1.7,0) -- (2.7,0);
    \draw[fill] (2.7,0) circle(0.07);
    \draw[thin] (3,0) circle(0.3);
    \node[right] at (3.3,0) {$h$};
\end{tikzpicture}
\end{center}
Here $G_1$ and $G_2$ are unbalanced graphs of genus one and $h$ is an even loop. There are two cones in the decomposition:
\[
\{e\leq f+g\},\quad\{f+g\leq e\}.
\]

\subsubsection*{Three $\FS_2$-bonds, no $\FS_3$-bond} The graph $G$ has the following form:
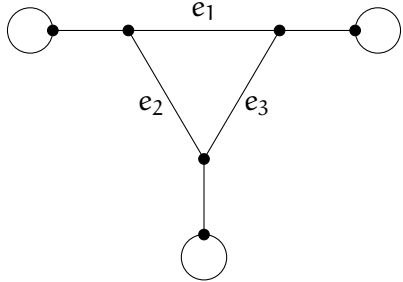
\begin{center}
\begin{tikzpicture}
    \draw[thin] (-1,0) -- (1,0) -- (0,-1.7) -- (-1,0);
    \node[above] at (0,0) {$e_1$};
    \node at (0.7,-1) {$e_3$};
    \node at (-0.7,-1) {$e_2$};
    \draw[fill] (-1,0) circle(0.07);
    \draw[fill] (1,0) circle(0.07);
    \draw[fill] (0,-1.7) circle(0.07);
    \draw[thin] (-1,0) -- (-2,0);
    \draw[thin] (1,0) -- (2,0);
    \draw[thin] (0,-1.7) -- (0,-2.7);
    \draw[fill] (-2,0) circle(0.07);
    \draw[fill] (2,0) circle(0.07);
    \draw[fill] (0,-2.7) circle(0.07);
    \draw[thin] (-2.3,0) circle(0.3);
    \draw[thin] (2.3,0) circle(0.3);
    \draw[thin] (0,-3) circle(0.3);
\end{tikzpicture}
\end{center}
The three loops are odd. There are four cones in the decomposition:
\[
\{e_1+e_2\leq e_3\},\quad \{e_1+e_3\leq e_2\},\quad \{e_2+e_3\leq e_1\},\quad \{e_i\leq e_j+e_k\mbox{ for all triples }i,j,k\}.
\]
We note that contracting all edges except the $e_i$ produces the $\mathrm{DR}_3$ example of Vologodsky~\cite{2004Vologodsky}, whose cone (which is three-dimensional) has the same decomposition.

\subsubsection*{No $\FS_2$-bonds, one $\FS_3$-bond} There is a unique graph $G$ with this property:
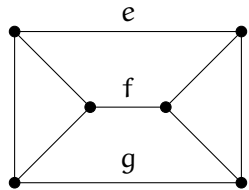
\begin{center}
\begin{tikzpicture}
    \draw[thin] (0,0) -- (3,0) -- (3,2) -- (0,2) -- (0,0);
    \draw[thin] (0,0) -- (1,1) -- (2,1) -- (3,0);
    \draw[thin] (0,2) -- (1,1);
    \draw[thin] (2,1) -- (3,2);
    \draw[fill] (0,0) circle(0.07);
    \draw[fill] (3,0) circle(0.07);
    \draw[fill] (3,2) circle(0.07);
    \draw[fill] (0,2) circle(0.07);
    \draw[fill] (1,1) circle(0.07);
    \draw[fill] (2,1) circle(0.07);
    \node[above] at (1.5,2) {$e$};
    \node[above] at (1.5,1) {$f$};
    \node[above] at (1.5,0) {$g$};
\end{tikzpicture}
\end{center}
The $\FS_3$-bond is $\{e,f,g\}$ and both triangles are unbalanced. There are three cones in the decomposition:
\[
\{e\leq \min(f,g)\},\quad \{f\leq \min(e,g)\},\quad \{g\leq \min(e,f)\}.
\]

\subsubsection*{One $\FS_2$-bond, one $\FS_3$-bond} The two sets always have exactly one edge in common, so we label them $\{e,f\}$ and $\{e,g,h\}$. The graph $G$ has the following form:
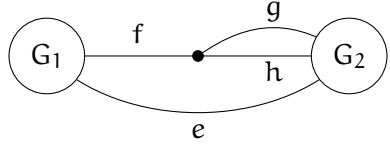
\begin{center}
\begin{tikzpicture}
    \draw[thin] (0,0) -- (4,0);
    \draw[thin] (2,0) .. controls (2.7,0.5) and (3.3,0.5) .. (4,0);
    \draw[thin] (0,0) .. controls (1,-1) and (3,-1) .. (4,0);
    \node at (1.2,0.3) {$f$};
    \node at (3,0.6) {$g$};
    \node at (3,-0.2) {$h$};
    \node at (2,-1) {$e$};
    \draw[fill] (2,0) circle(0.07);
    \draw[fill,white] (0,0) circle(0.5);
    \draw[thin] (0,0) circle(0.5);
    \node at (0,0) {$G_1$};
    \draw[fill,white] (4,0) circle(0.5);
    \draw[thin] (4,0) circle(0.5);
    \node at (4,0) {$G_2$};    
\end{tikzpicture}
\end{center}
Here $G_1$ and $G_2$ are unbalanced graphs of genus one. There are four cones in the decomposition: the cone $\{e\leq f\}$, and the cone $\{f\leq e\}$ that is further subdivided into three subcones:
\[
\{e-f\leq \min(g,h)\},\quad \{g\leq \min(e-f,h)\},\quad \{h\leq \min(e-f,g)\}.
\]

\subsubsection*{Two $\FS_2$-bonds, one $\FS_3$-bond} The three sets have a unique common edge and are otherwise disjoint, so we label them $\{e,f_1\}$, $\{e,f_2\}$, and $\{e,g,h\}$. The graph $G$ has the following form:
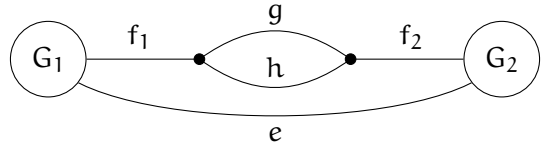
\begin{center}
\begin{tikzpicture}
    \draw[thin] (0,0) -- (2,0);
    \draw[thin] (2,0) .. controls (2.7,0.5) and (3.3,0.5) .. (4,0);
    \draw[thin] (2,0) .. controls (2.7,-0.5) and (3.3,-0.5) .. (4,0);
    \draw[thin] (4,0) -- (6,0);
    \draw[thin] (0,0) .. controls (0.7,-1) and (5.3,-1) .. (6,0);
    \draw[fill] (2,0) circle(0.07);
    \draw[fill] (4,0) circle(0.07);
    \node at (1.2,0.3) {$f_1$};
    \node at (3,0.6) {$g$};
    \node at (3,-0.1) {$h$};
    \node at (4.8,0.3) {$f_2$};
    \node at (3,-1) {$e$};
    \draw[fill,white] (0,0) circle(0.5);
    \draw[thin] (0,0) circle(0.5);
    \node at (0,0) {$G_1$};
    \draw[fill,white] (6,0) circle(0.5);
    \draw[thin] (6,0) circle(0.5);
    \node at (6,0) {$G_2$};    

\end{tikzpicture}
\end{center}
Here $G_1$ and $G_2$ are unbalanced graphs of genus one, and the cycle formed by $g$ and $h$ is balanced. There are four cones in the decomposition: the cone $\{e\leq f_1+f_2\}$, and the cone $\{f_1+f_2\leq e\}$ that is further subdivided into three subcones:
\[
\{e-f_1-f_2\leq \min(g,h)\},\quad \{g\leq \min(e-f_1-f_2,h)\},\quad \{h\leq \min(e-f_1-f_2,g)\}.
\]

\subsubsection*{Three $\FS_2$-bonds, one $\FS_3$-bond}  The graph $G$ has the form
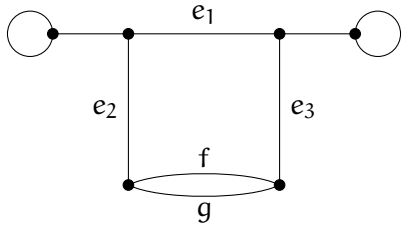
\begin{center}
    \begin{tikzpicture}
        \draw[thin] (0,0) -- (0,2) -- (2,2) -- (2,0);
        \draw[thin] (-1,2) -- (0,2);
        \draw[thin] (2,2) -- (3,2);
        \draw[thin] (-1.3,2) circle(0.3);
        \draw[thin] (3.3,2) circle(0.3);
        \draw[fill] (-1,2) circle(0.07);
        \draw[fill] (0,2) circle(0.07);
        \draw[fill] (2,2) circle(0.07);
        \draw[fill] (3,2) circle(0.07);
        \draw[fill] (0,0) circle(0.07);
        \draw[fill] (2,0) circle(0.07);
        \draw[thin] (0,0) .. controls (0.4,0.2) and (1.6,0.2) .. (2,0);
        \draw[thin] (0,0) .. controls (0.4,-0.2) and (1.6,-0.2) .. (2,0);
        \node[above] at (1,2) {$e_1$};
        \node[above] at (1,0.1) {$f$};
        \node[below] at (1,-0.1) {$g$};
        \node[left] at (0,1) {$e_2$};
        \node[right] at (2,1) {$e_3$};
    \end{tikzpicture}
\end{center}
The two loops and the cycle $\{f,g\}$ are odd. The $\FS_2$-bonds are $\{e_1,e_2\}$, $\{e_1,e_3\}$, and $\{e_2,e_3\}$, and the $\FS_3$-bond is $\{e_1,f,g\}$. There are six cones in the decomposition: the cone $\{e_2+e_3\leq e_1\}$ that is further subdivided into three subcones
\[
\{e_1-e_2-e_3\leq \min(f,g)\},\quad \{f\leq \min(e_1-e_2-e_3,g)\},\quad \{g\leq \min(e_1-e_2-e_3,f)\},
\]
and the cones
\[
\{e_1+e_2\leq e_3\},\quad \{e_1+e_3\leq e_2\},\quad \{e_i\leq e_j+e_k\mbox{ for all triples }i,j,k\}.
\]

\subsubsection*{One $\FS_2$-bond, two $\FS_3$-bonds} The $\FS_3$-bonds are always disjoint, and each contains exactly one element from the $\FS_2$-bond, so we denote the sets by $\{e_1,e_2\}$, $\{e_1,f_1,g_1\}$, and $\{e_2,f_2,g_2\}$. The graph $G$ has the following form:
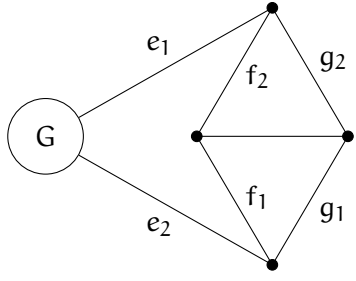
\begin{center}
\begin{tikzpicture}
    \draw[thin] (0,0) -- (3,1.7) -- (2,0) -- (3,-1.7) -- (0,0);
    \draw[thin] (3,1.7) -- (4,0) -- (3,-1.7);
    \draw[thin] (2,0) -- (4,0);
    \draw[fill] (3,1.7) circle(0.07);
    \draw[fill] (2,0) circle(0.07);
    \draw[fill] (3,-1.7) circle(0.07);
    \draw[fill] (4,0) circle(0.07);
    \draw[fill,white] (0,0) circle(0.5);
    \draw[thin] (0,0) circle(0.5);
    \node at (0,0) {$G$};
    \node at (1.5,-1.2) {$e_2$};
    \node at (1.5,1.2) {$e_1$};
    \node at (2.8,0.8) {$f_2$};
    \node at (2.8,-0.8) {$f_1$};
    \node at (3.8,1) {$g_2$};
    \node at (3.8,-1) {$g_1$};
\end{tikzpicture}
\end{center}
Here $G$ is an unbalanced graph of genus one, and both triangles on the right are unbalanced. There are six cones in the decomposition: the cone $\{e_1\leq e_2\}$ that is further subdivided into three subcones
\[
\{e_2-e_1\leq \min(f_2,g_2)\},\quad \{f_2\leq \min(e_2-e_1,g_2)\},\quad \{g_2\leq \min(e_2-e_1,f_2)\},
\]
and the cone $\{e_2\leq e_1\}$ that is likewise subdivided by exchanging the indices 1 and 2.

\subsubsection*{Two $\FS_2$-bonds, two $\FS_3$-bonds} The $\FS_2$-bonds are disjoint, while the remaining five pairwise intersections are distinct edges, therefore we label the sets as $\{e_1,f_1\}$, $\{e_2,f_2\}$, $\{e_1,f_2,g\}$, and $\{e_2,f_1,g\}$. The graph $G$ has the following form:
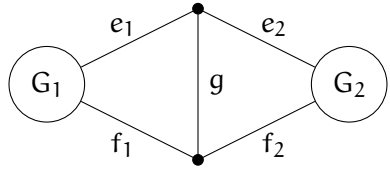
\begin{center}
\begin{tikzpicture}
    \draw[thin] (0,0) -- (2,1) -- (4,0);
    \draw[thin] (0,0) -- (2,-1) -- (4,0);
    \draw[thin] (2,1) -- (2,-1);
    \draw[fill,white] (0,0) circle(0.5);
    \draw[thin] (0,0) circle(0.5);
    \node at (0,0) {$G_1$};
    \draw[fill,white] (4,0) circle(0.5);
    \draw[thin] (4,0) circle(0.5);
    \node at (4,0) {$G_2$};
    \node[above] at (1,0.5) {$e_1$};
    \node[below] at (1,-0.5) {$f_1$};
    \node[above] at (3,0.5) {$e_2$};
    \node[below] at (3,-0.5) {$f_2$};
    \node[right] at (2,0) {$g$};
    \draw[fill] (2,1) circle(0.07);
    \draw[fill] (2,-1) circle(0.07);
\end{tikzpicture}
    \end{center}
Here $G_1$ and $G_2$ are unbalanced graphs of genus one. There are eight cones in the decomposition. First, there are two cones
\[
\{e_1\leq f_1,e_2\leq f_2\},\quad \{f_1\leq e_1,f_2\leq e_2\}.
\]
The cone $\{e_1\leq f_1,f_2\leq e_2\}$ is further subdivided into three subcones
\[
\{e_1-f_1\leq \min(f_2-e_2,g)\},\quad \{e_2-f_2\leq \min(e_1-f_1,g)\},\quad \{g\leq \min(e_1-f_1,f_2-e_2)\},
\]
and the cone $\{f_1\leq e_1,e_2\leq f_2\}$ is likewise subdivided by exchanging the indices 1 and 2.

\subsubsection*{Three $\FS_2$-bonds, two $\FS_3$-bonds} The graph $G$ has the form
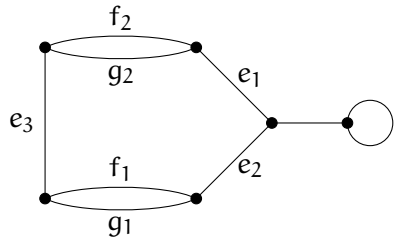
\begin{center}
    \begin{tikzpicture}
        \draw[thin] (0,0) -- (0,2);
        \draw[thin] (2,2) -- (3,1) -- (2,0);
        \draw[thin] (3,1) -- (4,1);
        \draw[thin] (0,0) .. controls (0.4,0.2) and (1.6,0.2) .. (2,0);
        \draw[thin] (0,0) .. controls (0.4,-0.2) and (1.6,-0.2) .. (2,0);
        \draw[thin] (0,2) .. controls (0.4,2.2) and (1.6,2.2) .. (2,2);
        \draw[thin] (0,2) .. controls (0.4,1.8) and (1.6,1.8) .. (2,2);
        \draw[fill] (0,0) circle(0.07);
        \draw[fill] (2,0) circle(0.07);
        \draw[fill] (0,2) circle(0.07);
        \draw[fill] (2,2) circle(0.07);
        \draw[fill] (3,1) circle(0.07);
        \draw[fill] (4,1) circle(0.07);
        \draw[thin] (4.3,1) circle(0.3);
        \node[right] at (2.4,1.6) {$e_1$};
        \node[above] at (1,0.1) {$f_1$};
        \node[below] at (1,-0.1) {$g_1$};
        \node[right] at (2.4,0.4) {$e_2$};
        \node[above] at (1,2.1) {$f_2$};
        \node[below] at (1,1.9) {$g_2$};
        \node[left] at (0,1) {$e_3$};

    \end{tikzpicture}
\end{center}
The loop and the cycles $\{f_1,g_1\}$ and $\{f_2,g_2\}$ are odd. The $\FS_2$-bonds are $\{e_1,e_2\}$, $\{e_1,e_3\}$, and $\{e_2,e_3\}$, and the $\FS_3$-bonds are $\{e_1,f_1,g_1\}$ and $\{e_2,f_2,g_2\}$. There are eight cones in the decomposition. The cone $e_2+e_3\leq e_1$ is further subdivided into three subcones
\[
\{e_1-e_2-e_3\leq \min(f_1,g_1)\},\quad \{f_1\leq \min(e_1-e_2-e_3,g_1)\},\quad \{g_1\leq \min(e_1-e_2-e_3,f_1)\}.
\]
The cone $e_1+e_3\leq e_2$ is likewise subdivided into three subcones by exchanging the indices 1 and 2. Finally, there are two more cones:
\[
\{e_1+e_2\leq e_3\},\quad \{e_i\leq e_j+e_k\mbox{ for all triples }i,j,k\}.
\]

\subsubsection*{Three $\FS_2$-bonds, three $\FS_3$-bonds} The graph $G$ has the form
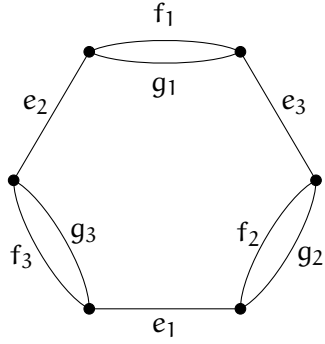
\begin{center}
    \begin{tikzpicture}
        \draw[thin] (0,0) -- (2,0);
        \draw[thin] (3,1.7) -- (2,3.4);
        \draw[thin] (0,3.4) -- (-1,1.7);
        \draw[thin] (0,3.4) .. controls (0.4,3.6) and (1.6,3.6) .. (2,3.4);
        \draw[thin] (0,3.4) .. controls (0.4,3.2) and (1.6,3.2) .. (2,3.4);
        \draw[thin] (2,0) .. controls (2,0.5) and (2.6,1.5) .. (3,1.7);
        \draw[thin] (2,0) .. controls (2.4,0.2) and (3,1.2) .. (3,1.7);
        \draw[thin] (0,0) .. controls (0,0.5) and (-0.6,1.5) .. (-1,1.7);
        \draw[thin] (0,0) .. controls (-0.4,0.2) and (-1,1.2) .. (-1,1.7);
        \draw[fill] (0,0) circle(0.07);
        \draw[fill] (2,0) circle(0.07);
        \draw[fill] (0,3.4) circle(0.07);
        \draw[fill] (2,3.4) circle(0.07);
        \draw[fill] (3,1.7) circle(0.07);
        \draw[fill] (-1,1.7) circle(0.07);
        \node[below] at (1,0) {$e_1$};
        \node[above] at (1,3.6) {$f_1$};
        \node[below] at (1,3.2) {$g_1$};
        \node[left] at (-0.4, 2.7) {$e_2$};
        \node[left] at (2.4, 1) {$f_2$};
        \node[right] at (2.6, 0.7) {$g_2$};
        \node[right] at (2.4, 2.7) {$e_3$};
        \node[left] at (-0.6, 0.7) {$f_3$};
        \node[right] at (-0.4, 1) {$g_3$};
    \end{tikzpicture}
\end{center}
The cycles $\{f_1,g_1\}$, $\{f_2,g_2\}$, and $\{f_3,g_3\}$ are odd, the $\FS_2$-bonds are $\{e_1,e_2\}$, $\{e_1,e_3\}$, and $\{e_2,e_3\}$, and the $\FS_3$-bonds are $\{e_1,f_1,g_1\}$, $\{e_2,f_2,g_2\}$, and $\{e_3,f_3,g_3\}$. There are ten cones in the decomposition. The cone $\{e_1+e_2\leq e_3\}$ is further subdivided into three subcones
\[
\{e_3-e_1-e_2\leq \min(f_3,g_3)\},\quad \{f_3\leq \min(e_3-e_1-e_2,g_3)\},\quad \{g_3\leq \min(e_3-e_1-e_2,f_3)\}.
\]
The two cones $\{e_2+e_3\leq e_1\}$ and $\{e_3+e_1\leq e_2\}$ are likewise subdivided into three cones each, by cyclically permuting the indices. Finally, there is a remaining cone 
\[
\{e_i\leq e_j+e_k\mbox{ for all triples }i,j,k\}.
\]

\bibliographystyle{amsalpha}
\bibliography{references}{}

\end{document}